\documentclass{article}
 \newcommand {\theoremstyle} [1] { }
\newenvironment{proof}{{\noindent\it\underline{Proof}}:}{\hfill$\Box$}
 
\newenvironment{po28}
{\noindent\it\underline{Proof of Theorem \ref{shooting}}:\rm}{\hfill$\Box$} \newenvironment{pt11}
{\noindent\it\underline{Proof of Theorem \ref{multi}}:\rm}{}

\newtheorem{thm}{Theorem}[section]
 \newtheorem{prop}{Proposition}[section]
 \theoremstyle{plain}
 \newtheorem{lem}[thm]{Lemma} %%Delete [thm] to re-start numbering
  
 \newtheorem{cor}[thm]{Corollary}
 \theoremstyle{definition}

 \theoremstyle{remark}
 \newtheorem{rem}[thm]{Remark}
 
\usepackage {amssymb}

\def\R{\mathbb{R}}

\def \ee{\varepsilon}

\usepackage {amssymb}
\usepackage{amsmath,amsfonts, color}
\usepackage{graphicx}

\makeatother

\begin{document}

\author{P. Amster and M. P. Kuna}

\title{On exact multiplicity for a second order equation with radiation boundary conditions}
\date{}
\maketitle

\begin{center}
Departamento de Matem\'atica, \\
Facultad de Ciencias Exactas y Naturales\\
Universidad de Buenos Aires
and \\
IMAS - CONICET\\
Ciudad Universitaria, Pabell\'on I,
(1428) Buenos Aires, Argentina \\
{\sl E-mails}: pamster@dm.uba.ar -- mpkuna@dm.uba.ar
\end{center}
\bigskip

\begin{abstract}
A second order ordinary differential equation with 
a superlinear term $g(x,u)$  under radiation boundary conditions is studied.
Using a shooting argument, all the results obtained in the previous work \cite{AKR3} for a Painlev\'e II equation are extended. 
It is proved that the uniqueness or multiplicity 
of solutions depend on the interaction between the mapping
$\frac {\partial g}{\partial u}(\cdot,0)$  
and the first eigenvalue of the associated linear operator. 
Furthermore, two open problems posed in \cite{AKR3}
regarding, on the one hand, the existence of sign-changing solutions and, on the other hand, exact multiplicity are solved. 

\medskip



\end{abstract}

\section{Introduction}

In \cite{AKR3}, the following problem 
arising on a two-ion electro-diffusion model 
(see \cite{B}, \cite{GCh}) 
was studied:
\begin{equation}
 \label{concreto}
 u''(x)=Ku(x)^3 + L(x)u(x)+A
\end{equation}
with
\begin{equation}
 \label{radiation}
u'(0)=a_0 u(0), \;\; u'(1)=a_1 u(1).
\end{equation}
Here, $K$ and $A$ 
some given positive constants and $L(x):= a_0^2 + (a_1^2-a_0^2)x$. Unlike the standard Robin condition, 
both coefficients $a_0$ and $a_1$ in the radiation boundary condition (\ref{radiation}) are assumed to be positive. 

It was proven that the problem has 
a negative solution; moreover, 
if $a_1\leq a_0$ then there are no other solutions. 
When $a_1 >a_0$, the solution is still unique for $A\gg 0$ but, 
if $A$ is sufficiently small, then
the problem has at least three solutions. 
Numerical evidence in \cite{AKR3} suggests that the number of solutions 
cannot be arbitrarily large and it was proven that, indeed, 
there exists exactly one negative solution, at most two positive solutions and 
that the set of solutions is bounded. 
It was conjectured that the maximum number of solutions is $3$ (typically, one of them negative and the other two positive)
but, however, none of the results 
in \cite{AKR3} prevents against the existence of many sign-changing solutions. 

%\smallskip
%
%Problem (\ref{concreto})-(\ref{radiation})
%arises on a model in two-ion electrodiffusion 
%derived independently by Grafov and Chernenko \cite{GCh} and Bass \cite{B} in the context of charged ion transport. 
%Two-point Dirichlet and periodic boundary value problems (BVPs) 
%for this Painlev\'e II equation and a 
%non-integrable generalization 
%were successively investigated in \cite{AMR1} and \cite{AMR2}. 
%A Neumann BVP for a Painlev\'e II equation depending on the 
%Dirichlet boundary values of the solution 
%was recently studied in \cite{AKR1} and \cite{AKR2} 
%by a two-dimensional shooting method.

In this work, we study a generalization of the previous problem, namely the equation 

\begin{equation}
 \label{eq}
u''(x)=g(x,u(x))+p(x) 
\end{equation}
where $p\in C([0,1])$ and $g:[0,1]\times \mathbb R \to \mathbb R$ 
is continuous, 
of class $C^1$ with respect to $u$ and superlinear, that is: 
\begin{equation}
\label{superlin} 
\lim_{ \left|u\right| \rightarrow + \infty} \frac{g(x,u)}{u}= +\infty
\end{equation}
uniformly in $x \in [0,1]$.  Without loss of generality, we shall assume that $g(x,0)=0$ for all $x\in [0,1]$. 
%In some cases we shall employ, instead, the slightly weaker condition: for every $M>0$ there exists $K>0$ such that 
%\begin{equation}
%\label{superlin-variat} 
%G(x,u)\ge Mu^2 - K
%\end{equation}
%for all $u$, where $G(x,u):=\int_0^u g(x,s)\, ds$. 
As before, we look for those solutions satisfying
the radiation boundary condition (\ref{radiation}) with $a_0$, $a_1 >0$. In the spirit of \cite{AKR3}, we shall 
assume throughout the paper that 
\begin{equation}\label{crec}
g \,\hbox{ is strictly increasing in }\, u
\end{equation}
and
\begin{equation}
\label{a-posit}
p(x) > 0 \,\hbox{ for all $x$}.
\end{equation}
For  general multiplicity results avoiding conditions (\ref{crec}) and (\ref{a-posit}) see  \cite{AK}. 
In the present setting, we shall demonstrate that all the results in \cite{AKR3} can be retrieved in a simple manner; 
furthermore, we shall give an answer to two questions that were left open. 
Specifically, it shall be seen that the set of solutions is 
bounded and contains always a negative solution, which tends uniformly to $-\infty$ as $p\to +\infty$ uniformly. Moreover, we shall extend the uniqueness statement in \cite{AKR3} 
by imposing the condition 
that $-\frac{\partial g}{\partial u}(x,u)$
is smaller than the first eigenvalue $\lambda_1$ of the associated linear operator for all $u$. Under a weaker condition, it shall be proved that  uniqueness holds also if $p$ is large.
As a complement of the uniqueness results, we shall also 
prove that if 
$-\frac{\partial g}{\partial u}(x,0)$ lies below $\lambda_1$ then the problem has at least three solutions, provided that $\|p\|_\infty$ is small. 
Furthermore, under 
an extra condition, which is fulfilled in (\ref{concreto}),  the multiplicity result is sharp.  
This 
extends 
the corresponding result for the particular problem 
(\ref{concreto}) and gives an answer to a  
question, sustained by numerical evidence but not 
proven in \cite{AKR3}:
 
\begin{thm}
\label{multi}
Assume that (\ref{superlin}), (\ref{crec}) and (\ref{a-posit}) hold. Then 
(\ref{eq})-(\ref{radiation}) has a negative solution. 
Moreover, if 
 \begin{equation}
  \label{H_9}
 \frac{\partial g}{\partial u}(\cdot,0) \lneq -\lambda_1
 \end{equation}
then there exists a constant $p_1 >0$ such that 
problem (\ref{eq})-(\ref{radiation}) has at least three solutions, one of them negative, one of them positive and another one sign-changing, when $\| p\|_{\infty}<  p_1$. 
If furthermore
$$
\frac{\partial g}{\partial u}(x,u)>\frac{g(x,u)}{u} 
$$
for all $u\neq 0$ and all $x$, 
then (\ref{eq})-(\ref{radiation}) has no other solutions, provided that $p_1$ is small enough.

\end{thm}

It follows that, under the previous assumptions, the number of solutions 
moves from $3$ to $1$ as  $\|p\|_\infty$ gets large. 
Similarly, for each fixed $p$, if we take $a_1$  as a parameter then 
uniqueness or multiplicity of solutions vary according to 
its different values. 
In general terms, multiplicity arises when $a_1$ is sufficiently large and should not be expected if $a_1$ is small. More precisely:

\begin{thm}
 \label{shooting} 

Assume that (\ref{superlin}), (\ref{crec}) and (\ref{a-posit}) hold. Then there exist constants $a^*>a_*>0$ such 
that:

\begin{enumerate}
\item 
If $a_1 > a ^*$ then 
problem (\ref{eq})-(\ref{radiation}) 
has at least three solutions, 
one of them negative and another one 
sign-changing.  

\item 
If $0<a_1 < a ^*$ then problem (\ref{eq})-(\ref{radiation}) 
has a unique (negative) solution.

\end{enumerate}

\end{thm}

The paper is organized as follows. Section \ref{main} is devoted 
to present several general aspects of the problem and state uniqueness and related results. 
In section \ref{sec-shoot}, we define a shooting-type operator that will be used 
to derive the proofs of Theorems \ref{multi} and \ref{shooting}. 
Some open questions are briefly exposed in a last section.

\section{Uniqueness and related results}

\label{main}

This section is devoted to introduce general results concerning problem  (\ref{eq})-(\ref{radiation}) that shall be used in the proofs of the main results. In the first place, we observe that solutions are bounded: 

\begin{thm}
\label{bound}
Assume that (\ref{superlin}) holds. Then 
there exists a constant $C$ such that 
every solution $u$ of 
(\ref{eq})-(\ref{radiation}) satisfies $\|u\|_{C^2}\le C$.

\end{thm}

\begin{proof}
Let $u$ be a solution. Multiply the equation by $u$ and integrate to obtain
$$
a_1u(1)^2 - a_0u(0)^2 = \int^1_0 [u'(x) ^2  +  g(x,u(x))u(x)  
+ p(x)u(x)]\, dx.
$$
Setting 
$\varphi(t):= [(a_1-a_0)t+a_0]u(t) ^2$, it is seen that
$$a_1u(1)^2 - a_0u(0)^2=\int_0^1\varphi'(x)\, dx \le 
\frac C\ee \|u\|_{L ^2} ^2 + 
\ee\|u'\|_{L ^2} ^2
$$ 
for arbitrary $\ee>0$ and 
$C$ depending on $\ee$, $a_1$ and $b_1$. Choose for example 
$\ee =\frac 12$ and set $M> 2C + \frac 12$, then by superlinearity there exists a constant $K$ 
(depending only on 
$M$ and $\|A\|_{L^2}$) such
that
$$
\frac 12 \|u'\|_{L ^2} ^2 + 
2C\|u\|_{L ^2} ^2 \ge 
\|u'\|_{L ^2} ^2
+ M\|u\|_{L ^2} ^2 - K.
$$
This implies $\|u\|_\infty\le \|u\|_{H^1} \le \sqrt{2K}$ 
and the proof follows using (\ref{eq}).

\end{proof}

Next, we may state an uniqueness result in terms of the first eigenvalue $\lambda_1$ of the (self-adjoint) linear operator 
$-u''$ under the boundary conditions (\ref{radiation}). To this end, let us simply recall that, by the standard Sturm-Liouville theory, 
$\lambda_1$ can be computed as the minimum of $-\int_0^1u''u\, dx $
over all the smooth functions satisfying (\ref{radiation}) such that $\|u\|_{L^2}=1$. 

\begin{thm}
\label{uniq}
Assume there exists an 
interval $I\subset \mathbb R$ such that, for all $u\in I$, 
\begin{equation}
   \label{H_1} 
\hbox{$\frac{\partial g}{\partial u}(x,u) \geq -\lambda_1$ for all $x\in [0,1]$} 
   \end{equation}
and the inequality is strict for some $x$ independent of $u$. 
Then 
(\ref{eq})-(\ref{radiation}) has at most one solution $u$ such that $u(x)\in I$ for all $x$. 
\end{thm}

\begin{proof}
Let $u_1, u_2:[0,1]\to I$ be solutions of (\ref{eq})-(\ref{radiation}) and define $w:=u_1-u_2$, then $w$ satisfies the boundary condition and
$$
w''(x)=g(x,u_1(x))-g(x,u_2(x))
          = \frac{\partial g}{\partial u}(x,\xi(x)) w(x)$$
for some $\xi(x)$ between $u_1(x)$ and $u_2(x)$.
Fix an open interval $J\neq \emptyset$ 
such that $\frac{\partial g}{\partial u}(x,\xi(x)) > -\lambda_1$ for $x\in J$ and suppose $w\not\equiv 0$ in $J$, then 
$$
0=\int_0^1 \left(w''w - \frac{\partial g}{\partial u}(x,\xi(x))w^2\right)dx 
< \int_0^1 (w''w + \lambda_1 w^2)\, dx\le 0,
$$
because $\lambda_1$ is the first eigenvalue. 
This contradiction proves that $w\equiv 0$ over $J$ and consequently $w=0$. 

\end{proof}

\begin{rem}

{As shown in \cite{AK}, $\lambda_1$ is a strictly decreasing continuous  function of $a_1$ and, moreover, $\lambda_1\ge 0$ 
if and only if $a_1\le  \frac{a_0}{a_0+1}$. In particular, 
when (\ref{crec}) holds, the latter inequality is a sufficient condition for uniqueness. 
However, it was proved in \cite{AKR3}, 
the (sharp) condition for 
uniqueness in the particular case (\ref{concreto}) is weaker, namely: $a_1\le a_0$. 
This is due to fact that, in this specific case, 
it is verified that $\lambda_1 \ge -a_1^2$
and hence
$$-\lambda_1\le a_1^2\le L(x) + 3Ku^2 =
\frac{\partial g}{\partial u}(x,u).
$$ 
}

\end{rem}

The next result shows that the failure of (\ref{H_1}) does not necessarily imply multiplicity: this fact was already observed in \cite{AKR3}
where, as mentioned, it was proven the solution of (\ref{concreto})-(\ref{radiation})
is unique also when $A$ is large. 
The latter property can be easily deduced in the general case with $p\equiv A$ from the next two 
results. The first 
of them establishes that, for $p$ large, 
solutions are negative; the second one proves that, 
under suitable assumptions, there cannot be two solutions with the same sign.

\begin{thm}
\label{agrande} 
Let (\ref{superlin}) hold. Then
there exists $p_0$ such that, if
$p(x)\ge p_0$ for all $x\in [0,1]$,
then
all the solutions 
of  (\ref{eq})-(\ref{radiation}) are negative.  
\end{thm}

\begin{proof}
Due to the superlinearity of $g$, 
for each $M\ge 0$ we may define the quantity
$$N_M:= \inf_{x\in [0,1], u\geq 0}\{g(x,u)-Mu\}>-\infty.$$
Then 
\begin{equation}
 \label{cotaf}
 g(x,u)\ge Mu +N_M
\end{equation}
for all $u\geq 0$. 
Let $M>0$ to be determined, fix 
$p_0> -N_M$ and let
$u$ be a solution of (\ref{eq})-(\ref{radiation}) such that $u(x)\ge 0$ for some $x\in [0,1]$. 
In view of (\ref{cotaf}), the inequality
$ u''(x)\geq g(x,u(x))+p_0$
implies that 
\begin{equation}
 \label{eq2}
u''(x)>Mu(x)
\end{equation}
whenever $u(x)\ge 0$. 
We deduce that, if $x_0 \in [0,1]$ 
is such that $u(x_0)$ and $u'(x_0)$ are nonnegative, then
$u(x)$ and $u'(x)$ are strictly positive for $x>x_0$. 
Multiply (\ref{eq2}) by $u'$ and integrate to obtain, for $x>x_0$:
\begin{equation}
\label{cotax}
u'(x)^2 > u'(x_0)^2+M(u(x)^2-u(x_0)^2).
\end{equation}

If $u(0)> 0$, then $u'(0)>0$ and
$$u(1)^2-u(0)^2= \int_0^1 2u(x)u'(x)dx >2a_0 u(0)^2.$$
Thus, 
\begin{equation}
 \label{cotau1}
 u(1)^2-u(0)^2 > \frac{2a_0}{1+2a_0}u(1)^2
\end{equation}
and fixing $M=a_1^2\frac{1+2a_0}{2a_0}$  we obtain, from 
(\ref{cotax}) and (\ref{cotau1}):
$$a_1^2u(1)^2 > 
 M\frac{2a_0}{1+2a_0} u(1)^2 = a_1^2 u(1)^2.$$
This contradiction proves 
that there are no positive solutions when $p_0>-N_M$. 

On the other hand, if $u(0)\le 0$ 
then $u$ vanishes at a (unique) value $x_0$, with
$u'(x_0)\ge 0$.  
Fix $M=a_1^2$, then (\ref{cotax}) yields
$$ a_1^2u(1)^2=u'(1)^2 > u'(x_0)^2 + a_1^2u(1)^2\ge  a_1^2u(1)^2, 
$$
a contradiction.

\end{proof}

\begin{thm}
\label{unicanegativa}
Assume there exists an interval 
$I\subset \mathbb R_{\ne 0}$ such that
   \begin{equation}
      \hbox{\label{H_2} $\frac{\partial g}{\partial u}(x,u)> \frac{g(x,u) +p(x)}{u}$}
   \end{equation}
holds for all $x\in [0,1]$ and $u\in I$.
Then there exists at most one solution $u$ of (\ref{eq})-(\ref{radiation}) such that $u(x)\in I$ for all $x$.
\end{thm}

\begin{proof}
Let $u_1, u_2:[0,1]\to I$ be two different solutions, then $u_1(0)\neq u_2(0)$. Suppose for example that $u_1<u_2$ 
over $[0,x_0)$, then 
$$
u_1''(x) = \frac{g(x,u_1(x))+p(x)}{u_1(x)} u_1(x) < 
\frac{g(x,u_2(x))+p(x)}{u_2(x)} u_1(x)
$$
and hence
$$
u_1''(x)u_2(x) > u_1(x)u_2''(x)\qquad x<x_0.
$$
We conclude that 

\begin{equation}
\label{x0}
u_1'(x_0)u_2(x_0) 
> u_1(x_0)u_2'(x_0),
\end{equation}
and a contradiction yields if 
$x_0=1$. Thus, we may suppose that $u_1$ and $u_2$ meet for the first time at $x_0$, then $u_1(x_0)=u_2(x_0)$ and 
$u_1'(x_0)\ge u_2'(x_0)$. Again, this contradicts 
(\ref{x0}).

\end{proof}

\begin{rem}
Condition (\ref{H_2}) implies that the function $\frac{g(x,u)+p(x)}u$ increases in $u$ when $I\subset \mathbb R_+$ and decreases when $I\subset \mathbb{\R}_-$. 
Moreover, if $0\in \partial I$ then 
$sp\le 0$, where $s$ denotes the 
sign of the elements of $I$. 
In particular, if the condition 
holds for all $u\neq 0$, then $p=0$. This case 
is well known in the literature (see e.g. \cite{CCN}) and implies 
that if $u_0\ne 0$ is a critical point of the 
associated functional $\mathcal J$, then $u_0$ is 
transversal to the Nehari 
manifold introduced after the pioneering work \cite{N}, namely: 
$$\mathcal N:=\{u\in H^1(0,1)\setminus\{0\} :D\mathcal J(u)(u) =0\}.$$ Indeed, setting 
$\mathcal I(u):= D\mathcal J(u)(u)$ it is readily seen that $T_{u_0}\mathcal N= ker(D\mathcal I(u_0))$ and  
$D\mathcal I(u_0)(u_0)>0$. 
For the particular case of problem (\ref{concreto}), condition (\ref{H_2}) simply reads $\frac A{u^3}<2K$, so the previous result applies with $I=(-\infty,0)$ and $I=\left(\sqrt[3]{\frac A{2K}},+\infty\right)$.

\end{rem}

The next theorem generalizes another 
result from \cite{AKR3}, 
concerning the behaviour of the solutions as $p$ increases. 
We know that all solutions are negative if $p\ge p_0\gg 0$ and it is readily verified (e.g. by the method of upper and lower solutions) that a solution always exists; however, if the assumptions of Theorem \ref{uniq} or  
Theorem \ref{unicanegativa} are not satisfied, then 
there might be more than one negative solution. 
As we shall see, all possible solutions tend uniformly to $-\infty$ as $p$ tends uniformly to $+\infty$. 
In order to emphasize the dependence on $p$, any solution shall be denoted $u_p$, despite the fact that it might not be unique.

\begin{thm}
\label{uniformly}
Assume that (\ref{superlin}) holds and let $u_p$ be a solution of (\ref{eq})-(\ref{radiation}).
Then 
$u_p \rightarrow -\infty$ uniformly when $p\rightarrow +\infty$ uniformly.
\end{thm} 
\begin{proof}
Let $p\ge p_0$ for some large constant $p_0$. From Theorem \ref{agrande}, we may assume $u_p<0$. 
Fix $x_p$ such that $\max_{x\in [0,1]} u_p(x)=u_p(x_p)$, then 
$x_p<1$. Suppose $u_p(x_p)>-M$ and fix $p_0$ large enough, such that,  
\begin{equation}
\label{alosufgrande}
g(x,u)+p_0 > Ma_0, \;\;\; \hbox{ for all }u\ge -(1+a_0)M.
\end{equation}
It follows that $x_p=0$. Consider the maximum value 
$\delta \leq 1$ such that $u_p''(x)\geq 0$ for all $x\in [0,\delta]$, then $u'_p(x)\ge u'_p(0)=a_0 u_p(0) >-Ma_0$, for $x\leq \delta$.
 Hence, $u_p(\delta) > u_p(0)-\delta M a_0 \geq -M(1+a_0)$ 
 and by (\ref{alosufgrande}) we conclude that 
 $u''_p(\delta)>0$. 
 Thus, $\delta =1$ 
  and, in particular, $u_p(x)> -M(1+a_0)$. 
Using (\ref{alosufgrande}) again, it follows that $u_p''(x) > Ma_0$ for all $x$. Then 
$u'_p(1) > u'_p(0)+ Ma_0>0$, a contradiction.

\end{proof}

Combining the previous result with 
Theorems \ref{uniq} and \ref{unicanegativa} we deduce that, in fact, the solution is typically unique when $p$ is large. 
Indeed, due to superlinearity we observe that, on the one hand, 
$\frac{\partial g}{\partial u}(x,u)$ cannot 
remain bounded from above as $u\to -\infty$ and, on the other hand,  the function $\frac{g(x,u)}{u}$ cannot be 
increasing in $u$ over any interval $(-\infty,C)$. In other words, it is reasonable to  expect that either 
condition (\ref{H_1}) holds 
or  $ \frac{\partial g}{\partial u}(x,u)-\frac{g(x,u)}u \ge \frac ku$
when $u\ll 0$. 
Any of these conditions, which are   
fulfilled in the particular case (\ref{concreto}), ensures the   applicability of  Theorems \ref{uniq} or \ref{unicanegativa}
when $p$ is large. Thus, the following corollary is obtained:

\begin{cor}
Assume that (\ref{superlin}) holds. Moreover, assume there exists $C\le 0$ such that one of the following conditions holds:

\begin{enumerate}
\item 
Condition (\ref{H_1}) 
holds for all $u\le C$,
\item 
$$\sup_{x\in[0,1],u\le C} u\frac{\partial g}{\partial u}(x,u)-g(x,u)<+\infty.$$
\end{enumerate}
Then there exists $p_0$ such that problem (\ref{eq})-(\ref{radiation}) has a unique solution, which is negative,   
for all $p\ge p_0$. 
\end{cor}

\begin{proof}
From Theorem 
\ref{uniformly}, there exists $\tilde p$ such that if $u$ is a solution for $p\ge \tilde p$ then $u(x)\le C$ for all $x$. 
If the first condition holds, then 
the proof follows directly from Theorem \ref{uniq}. 
Otherwise, there exists a constant $M$ such that 
$$
\frac{\partial g}{\partial u}(x,u)-\frac{g(x,u)}u > \frac{M}u$$
for all $u\le C$ and, by Theorem 
\ref{unicanegativa}, the result follows taking $p_0$ as the maximum value between $\tilde p$ and $M$. 
\end{proof}

\section{A shooting operator for problem (\ref{eq})-(\ref{radiation})}
\label{sec-shoot}

This section is devoted to proof Theorems \ref{multi} and \ref{shooting} by means of a shooting-type operator. To this end, let us firstly state the following 
lemma, which ensures that, if 
(\ref{crec}) holds,
then the graphs of two different solutions of (\ref{eq}) 
with initial 
condition $u'(0)=a_0u(0)$ do not intersect. More generally, 

\begin{lem}
\label{nosecruzan} 

Let $u_1$ and $u_2$ be solutions of (\ref{eq}) defined over an interval $[0,b]$ such that 
$u_1(0) >u_2(0)$ and $u_1'(0)>u_2'(0)$ and 
assume that (\ref{crec}) holds. 
Then $u_1>u_2$ and $u_1'>u_2'$ on $[0,b]$.    
\end{lem}

\begin{proof}
Set $u(x)=u_1(x)- u_2(x)$, then $u''(x)=\theta(x)u(x)$ on 
$[0,b)$, where $\theta(x):=
\frac{\partial g}{\partial u}(x,\xi(x))> 0$. Thus, the result follows since $u(0), u'(0) >0$. 
\end{proof}

\medskip
Next, we define our shooting operator as follows.
For each fixed $\lambda \in \mathbb R$, let $u_\lambda$ be the unique solution of problem
\begin{equation}
 \label{shoot}
 \left\{ \begin{array}{ll}
       u''(x)=g(x,u(x))+p(x)\\
       u(0)=\lambda, \;\; u'(0)=a_0\lambda
         \end{array}
\right.
\end{equation}
and define 
the function $T: \mathcal D \to \mathbb R$, by
$$T(\lambda)= \frac{u'_\lambda (1)}{u_\lambda (1)},$$
where $\mathcal D\subset \mathbb R$ 
is the set of values of $\lambda$ such that the corresponding solution $u_\lambda$ of (\ref{shoot}) is defined on $[0,1]$, with $u_\lambda(1)\ne 0$. 
Thus, solutions of (\ref{eq})-(\ref{radiation}) 
that do not vanish on $x=1$ can 
be characterized as the functions
$u_\lambda$, where $\lambda\in \mathcal D$ is such that  $T(\lambda)=a_1$. 
By continuity arguments, it is easy to verify that, 
for each $s\in \mathbb R$, there exists $\lambda$ 
such that $u_\lambda (1) = s$. 
By Lemma \ref{nosecruzan}, 
this value of $\lambda$ is unique; 
in particular, there exists a unique $\lambda_0$ such that $u_{\lambda_0}=0$.  
Thus, we conclude that 
$$\mathcal D= (\lambda_*, \lambda_0)\cup (\lambda_0, \lambda^*)$$
for some $\lambda_*\ge -\infty$ and $\lambda^* \le +\infty$.  

From (\ref{a-posit}), it follows that $\lambda_0<0$ and, furthermore:  
if $\lambda >0$ then $u_\lambda$ is positive and if 
$\lambda_0\le \lambda\le 0$ then $u_\lambda$ vanishes exactly once in $[0,1]$. In particular, $u_{\lambda_0}<0$ in $[0,1)$
and, since $u_{\lambda_0}''(x)>0$ when $x$ is 
close to $1$, we conclude that
 $u_{\lambda_0}'(1) >0$. Hence,

$$\lim_{\lambda \rightarrow \lambda_0^-} T(\lambda)= -\infty, \qquad \lim_{\lambda \rightarrow \lambda_0^+} T(\lambda)= +\infty.$$
We claim that also 
$$\lim_{\lambda \rightarrow (\lambda^*)^-} T(\lambda)= +\infty,\qquad \lim_{\lambda \rightarrow (\lambda_*)^+} T(\lambda)= +\infty.$$ 
Indeed, observe firstly that, because solutions of (\ref{shoot}) do not cross each other,  
$\lim_{\lambda \rightarrow (\lambda^*)^-} u_\lambda (1)=+\infty$.
On the other hand, multiplying (\ref{eq}) by $u'$
it is easy to see, given $M>0$ that 
$$\left| u'_\lambda (1)\right| \geq \sqrt M \mathcal O(\left| u_\lambda(1)\right|)$$
for $|\lambda|$ sufficiently large. 
This implies that 
$$\left|T(\lambda)\right|=\left|\frac{u'_\lambda (1)}{u_\lambda (1)}\right| >\sqrt M$$
and the claim follows.

The previous considerations show the 
existence of 
$\lambda_{min}\in (\lambda_0,\lambda^*)$ such 
that $T(\lambda_{min}) \le T(\lambda)$ 
for all 
$\lambda\in (\lambda_0,\lambda^*)$. 
The value $a_{min}:=T(\lambda_{min})>0$ depends on $p$ 
and, in this context, Theorem \ref{agrande} simply states that if $p\ge p_0$ for some large enough constant $p_0$ then
$a_{min} >a_1$. 
Also, we easily deduce some of
the conclusions of 
Theorems \ref{multi} and \ref{shooting}, as shown in the following figure. 

\bigskip

\qquad\quad \includegraphics[scale=0.5]{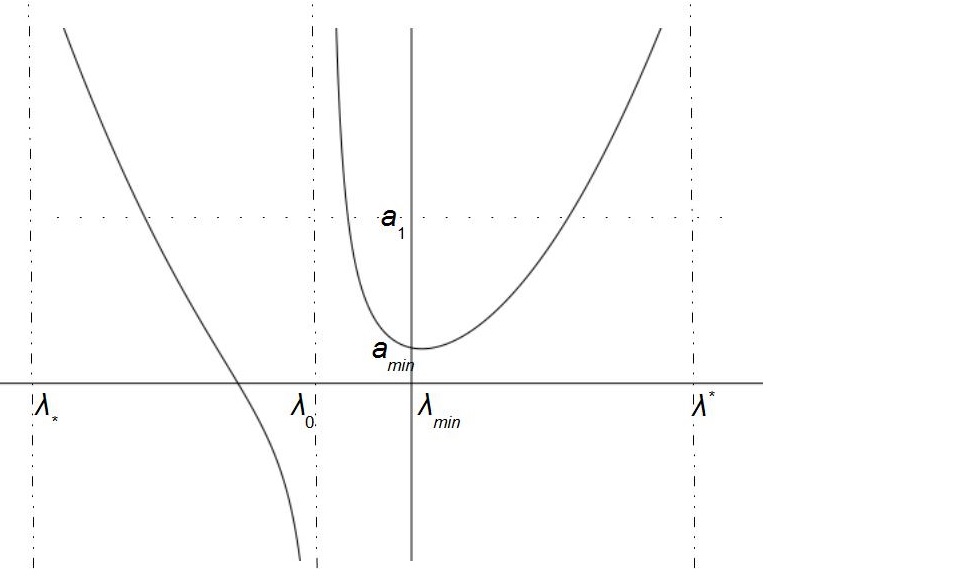}

\bigskip

By continuity, 
there exists $\lambda < \lambda_0$ such that
$T(\lambda)=a_1$; the 
corresponding $u_\lambda$ is a negative solution.
Uniqueness of negative solutions does not follow directly 
from this setting, unless an extra assumption like (\ref{H_2}) is assumed for $u<0$ (see Proposition \ref{monot} below). 
However, recall that if $a_1< \frac {a_0}{a_0+1}$, then 
$\lambda_1 >0$ so (\ref{H_1}) is satisfied; thus
uniqueness holds if $a_1$ is small.

\medskip 

\begin{po28} From the previous considerations, the problem has a negative solution, which is unique if $a_1$ is sufficiently small.  
Moreover, 
the equation $T(\lambda)=a_1$ 
has, over the interval $(\lambda_0, \lambda^*)$ 
at least two solutions 
when $a_1>a_{min}$ is large 
and no solutions 
when $a_1$ is small. Finally, observe that, 
as the value 
of $a_1$ increases, at least one of those solutions is located 
in $(\lambda_0,0)$. 

\end{po28}

\begin{rem}

If $\lambda^*>0$ or, equivalently, 
if $u_0$ is defined on $[0,1]$, then we deduce that the problem has also a positive solution when $a_1\gg 0$.

\end{rem}

Under appropriate conditions, a lower bound for $a_{min}$ is easily obtained as follows: 

\begin{prop}
Assume that 
(\ref{crec}) and (\ref{a-posit}) hold. 
If there exists $r\le a_0$ such that
$g(x,u) + p(x) >r^2u$ for all $u \ge 0$ and all $x$, 
then $a_{min}>r$. 

\end{prop}
\begin{proof}
 Fix $\lambda\in (\lambda_0,\lambda ^*)$ and 
let $v(x):= e^{rx}$. 
Define $x_0$ 
as the minimum value such that $u_\lambda$ 
is positive after $x_0$ 
and observe that
$$v(x)u_\lambda''(x) > v(x)r^2u_\lambda(x) >  
v(x)r ^2u_\lambda(x) = v''(x)u_\lambda(x)
$$
for $x>x_0$. 
Thus, 
$$
v(1)[u_\lambda'(1) - ru_\lambda(1)] > v(x_0)[u_\lambda'(x_0) - ru_\lambda(x_0)]\ge v(x_0)[u_\lambda'(x_0) - ru_\lambda(x_0)]\ge 0
$$
and we conclude that $T(\lambda)=\frac {u_\lambda'(1)}{u_\lambda(1)}>r$. 
\end{proof}

\begin{rem}
 For problem (\ref{concreto}), the previous proposition implies that $a_{min} > \min\{ a_0,a_1\}$, 
which provides an alternative proof of the fact that the problem has no positive nor sign-changing 
solutions when $a_1\le a_0$. 
\end{rem}

In order to complete the proof of Theorem 
\ref{multi}, let us make  
a more careful description of the graph of $T$. 
With this aim, compute
$$T'(\lambda)=\frac{\partial}{\partial \lambda}\left( \frac{u'_\lambda (1)}{u_\lambda (1)}\right)= \frac{u_\lambda (1) \frac{\partial u'_\lambda}{\partial \lambda}(1)-u'_\lambda (1)\frac{\partial u_\lambda}{\partial \lambda}(1)}{u_\lambda (1)^2}$$
and set $w_\lambda:=\frac{\partial u_\lambda}{\partial \lambda}$, then 
$$T'(\lambda)=\frac{u_\lambda (1) w_\lambda'(1)-u'_\lambda(1)w_\lambda(1)}{u_\lambda(1)^2}.$$
Moreover, observe that $w_\lambda$ solves the linear problem
\begin{equation}
 \label{w}
 \left\{ \begin{array}{ll}
          w_\lambda''(x)=\frac{\partial g}{\partial u}(x,u_\lambda(x)) w_\lambda(x) \\
          w_\lambda(0)=1, \;\; w_\lambda'(0)=a_0,
         \end{array}
\right.
\end{equation}
and hence
\begin{align}
\label{tprima}
&u_\lambda (1) w_\lambda'(1)-u'_\lambda(1)w_\lambda(1)
=\int_0^1 \left( u_\lambda (x) w_\lambda''(x)- u''_\lambda (x) w_\lambda(x)\right) dx=\nonumber\\
&=\int_0^1 \left( u_\lambda(x)\frac{\partial g}{\partial u}(x,u_\lambda(x))-g(x, u_\lambda(x))-p(x) \right) w_\lambda(x)dx.
\end{align}

Taking into account that 
$w_\lambda(x)>0$ for all $x$ and that $u_\lambda$ is negative for $\lambda<\lambda_0$ and positive for $\lambda>0$, the following proposition is obtained:

\begin{prop}
\label{monot}
Assume that (\ref{superlin}), (\ref{crec}) and (\ref{a-posit}) hold. Then: 
\begin{enumerate}
\item 
$T$ is strictly decreasing for 
$\lambda <\lambda_0$, provided that 
(\ref{H_2}) holds for $u<0$. 

\item 
$T$ is strictly increasing for 
$\lambda >C$, provided that 
(\ref{H_2}) holds for $u>C\ge 0$. 

\end{enumerate}

\end{prop}

\begin{rem}

In particular, the previous proposition shows that, when (\ref{crec}) and (\ref{a-posit}) are assumed, 
the conclusions of Theorem \ref{unicanegativa} are retrieved in a simple manner.

\end{rem}

Assume firstly that $p=0$.  
Although (\ref{a-posit}) obviously fails, the operator 
$T$ is well defined, with $\lambda_0=0$. 
Moreover, using the L'H\^opital rule we deduce that
$$\lim_{\lambda\to 0}T(\lambda)= \lim_{\lambda\to 0} \frac{u_\lambda'(1)}{u_\lambda(1)} = 
\lim_{\lambda\to 0} \frac{w_\lambda'(1)}{w_\lambda(1)} = 
\frac{\Phi'(1)}{\Phi(1)},
$$
where $\Phi:=w_0$, that is, 
the unique solution of the linear initial value problem 

\begin{equation}
\label{ivp}
\Phi''(x)= \frac{\partial g}{\partial u}(x,0)\Phi(x), 
\qquad \Phi'(0)=a_0\Phi(0)=a_0. 
\end{equation}
Thus, $T$ can be extended continuously 
to a positive function defined over 
$(\lambda_*,\lambda^*)$, which tends to $+\infty$ as $\lambda\to (\lambda_*)^+$ or $\lambda\to (\lambda^*)^-$.
Furthermore, if 
(\ref{H_2}) holds for $u\ne 0$ then it decreases strictly on $(\lambda_*,0)$ and increases strictly on $(0,\lambda ^*)$.

We are now in condition of completing the proof 
of Theorem \ref{multi}.  To this end, we shall need 
the following lemma:

\begin{lem}
\label{lemaphi}
Assume that
(\ref{crec}) and (\ref{H_9}) holds. Then 
$\Phi'(1)<a_1\Phi(1)$.
 \end{lem}

\begin{proof}
Let $\varphi_1$ be the (unique) eigenfunction corresponding to $\lambda_1$ such that $\varphi_1(0)=1$, then
it is readily verified that $\varphi_1(x)>0$ for all $x$. Moreover, it is seen from (\ref{ivp}) 
that also $\Phi(x)> 0$ for all 
$x$. 
Then 
$$\varphi_1(x)\Phi''(x) 
= \frac {\partial g}{\partial u}(x,0)\Phi(x)\varphi_1(x) 
\le -\lambda_1\Phi(x)\varphi_1(x)= 
\Phi(x)\varphi_1''(x)
$$ 
and the inequality is strict for some $x$. Integration yields
$$
\varphi(1)\Phi'(1) < \Phi(1)\varphi'(1)  = a_1\Phi(1)\varphi(1)
$$
and the proof follows.

\end{proof}

\begin{pt11}
In view of the previous Lemma, the proof is an immediate corollary of the following proposition, slightly more general: 

\end{pt11}

\begin{prop}
\label{exact-prop}
Assume that (\ref{superlin}), (\ref{crec}) and (\ref{a-posit}) hold and that 
$\Phi'(1) < a_1 \Phi(1)$. Then there exists a 
constant $p_1 >0$ such that problem (\ref{eq})-(\ref{radiation}) has at least three solutions when $\| p\|_{\infty}<  p_1$. Moreover, one of the solutions is negative, 
one of them positive and another one sign-changing. 
If furthermore
(\ref{H_2}) holds with $p=0$  for all $u\ne 0$, then 
there exists 
$p_1$ such that the problem has exactly three solutions, provided that $\|p\|_{\infty} < p_1$. 
Moreover, exactly one of the solutions 
is negative and another one changes sign. 
\end{prop}

\begin{proof}
From the previous considerations we know that, if $p$ is small, then $\lambda^*>0$ and
$T(0) < a_1$; thus, the existence of at least three solutions follows. Clearly, one of the solutions is negative, another one is positive and another one changes sign. 

From now on, assume that 
(\ref{H_2}) with $p=0$ holds for all $u\ne 0$.
Consider, for arbitrary $p$, the mapping
$$
R_p(\lambda) := u_\lambda'(1) - a_1u_\lambda(1).
$$
Let us firstly take $p=0$. 
From the previous computations, 
we know that 
$sgn(T'(\lambda))=sgn (u_\lambda)=sgn (\lambda)$ for $\lambda\neq 0$ and
$T(0)=\frac{\Phi'(1)}{\Phi(1)} <a_1$, whence
$R_0$ has 
exactly three roots $\{0, \lambda_\pm\}$ with 
$\lambda_-<0<\lambda_+$. 
Moreover, write as before 
$$T'(\lambda) =
\frac{u_\lambda (1) w_\lambda'(1)-u'_\lambda(1)w_\lambda(1)}{u_\lambda(1)^2} 
= \frac{w_\lambda(1)}{u_\lambda(1)}
\left(\frac{w_\lambda'(1)}{w_\lambda(1)} - 
T(\lambda)
\right) 
$$ 
to deduce that
$$\frac{w_{\lambda_\pm}'(1)}{w_{\lambda_\pm}(1)} 
> T(\lambda_\pm)=a_1.$$

Next, observe that 
$$R_0'(\lambda)= w_\lambda'(1) - a_1w_\lambda(1)= 
w_\lambda(1)\left(
\frac{w_\lambda'(1)}{w_\lambda(1)} -a_1\right), 
$$
so 
$R_0'(\lambda_\pm)>0$. On the other hand, $R_0'(0) = \Phi'(1) - a_1 \Phi(1) <0$ and, by continuity, we conclude that if $p$
is close to $0$ then $R_p$ has exactly three roots. 
Furthermore, $T(0)$ is close to $\frac{\Phi'(1)}{\Phi(1)}<a_1$, so the equation $T(\lambda)=a_1$ has at least one solution in 
$(\lambda_0,0)$. Finally, observe that if $p$ is small 
then $u_0$ is defined in $[0,1]$; thus, $\lambda^*>0$ and letting $p$ be smaller 
if necessary we conclude that the equation 
$T(\lambda)=a_1$ has also a solution in 
$(0,\lambda^*)$
\end{proof}

\begin{rem}
In particular, all the assumptions of the previous proposition are fulfilled for problem (\ref{concreto}) if (and only if) $a_1>a_0$. Indeed, in this 
case it is readily seen that $\lambda_1 < -a_1 ^2$
and hence 
$\frac{\partial g}{\partial u}(x,0)=L(x) \le a_1^2  < -\lambda_1$.

\end{rem}

\section{Open questions}

\begin{enumerate}
\item
Numerical experiments for the particular case (\ref{concreto}) suggest that 
$T'' >0$ for $\lambda>\lambda_0$. If this is true, then an exact multiplicity result yields for arbitrary $p$, depending on whether $a_{min}$ is 
smaller, equal or larger than $a_1$. It would be interesting to investigate if this fact could be verified for the general case, under appropriate conditions, 
using the differential equation for 
$z_\lambda:=\frac{\partial w_\lambda}{\partial\lambda}=\frac{\partial^2 u_\lambda}{\partial\lambda^2}$, namely
$$z_\lambda''(x) = \frac{\partial g}{\partial u}(x,u_\lambda(x)) z_\lambda(x) + \frac{\partial^2 g}{\partial u^2}(x,u_\lambda(x)) w_\lambda(x)^2, \qquad z_\lambda(0)=z_\lambda'(0)=0.
$$

\item 
Is it possible to obtain 
an exact multiplicity result also for $a_1$ 
large? Observe that, in such a case, 
the behaviour of $T$ can be controlled 
near $\lambda_0$, but it is not easy to see what happens as $\lambda$ gets closer to $\lambda_*$ or $\lambda^*$. 
In more precise terms, we may set $\varepsilon:=\frac 1{a_1}$ and 
$$R_\varepsilon(\lambda):= \varepsilon u_\lambda'(1) - u_\lambda(1).$$
Then $R_0(\lambda)=-u_\lambda(1)$ decreases from
$+\infty$ to $-\infty$ 
over $(\lambda_*, \lambda^*)$. Furthermore, 
$R_0'(\lambda)=-w_\lambda(1)<0$ for all $\lambda$; thus, if 
$\varepsilon$ is small, 
then $R_\varepsilon$ has, near $\lambda_0$,
a unique root. However, for 
$\varepsilon\ne 0$ the graph of $R_\varepsilon$ 
bends in such a way that it tends to $\pm \infty$ as $\lambda$ gets closer to $\lambda^*$ and $\lambda_*$ respectively. 
This ensures the existence of at least three solutions for $\varepsilon$ small, although there might be more. 
Clearly, there exists $\lambda_1$ such that 
$R_\varepsilon$ increases with $\varepsilon$ 
for $\lambda>\lambda_1$ and decreases when $\lambda<\lambda_1$; moreover, if $K\subset (\lambda_*,\lambda^*)$ 
is a compact neighborhood 
of $\lambda_0$, then $R_\varepsilon$
vanishes exactly once in $K$ when 
$\varepsilon=\varepsilon(K)$ is small. 
This is due to the fact that $R_\varepsilon$ tends to 
$R_0$ over $K$ for the $C^1$ norm. 
However, 
it is not clear which condition would be appropriate in order to prevent against a possible `strange' behaviour of $R_\varepsilon$ outside compact sets. For example, taking into account the 
superlinearity, we might impose the assumption that 
$\frac{\partial g}{\partial u}(x,u)$ tends uniformly to $+\infty$ as $|u|\to +\infty$. This would 
ensure that $R_\varepsilon$ has positive derivative 
near the endpoints of its domain but, 
still, it might change sign many times.

\item
How does the graph of $T$ vary with respect to $p$? 
Suppose for simplicity that $p$ is a constant and let 
$y_p:=\frac{\partial u_\lambda}{\partial p}$. 
Then 
$$y_p''=\frac{\partial g}{\partial u}(x,u_\lambda(x)) y_p+ 1\qquad y_p(0)=y_p'(0)=0
$$
and the sign of $\frac{\partial T}{\partial p}$ coincides with the sign of the integral
$$
\int_0^1 \left( u_\lambda(x)\frac{\partial g}{\partial u}(x,u_\lambda(x))-g(x, u_\lambda(x))-p \right) y_p(x) + u_\lambda(x)\, dx.
$$
If (\ref{H_2}) holds for $u<0$, then $\frac{\partial T}{\partial p} <0$ for $\lambda<\lambda_0$. In particular, the 
(unique) value $\lambda<\lambda_0$ 
for which $T(\lambda)=a_1$ 
moves to the left as $p$ increases. 
This is consistent with the fact 
that the negative 
solution tends uniformly to $-\infty$ as $p\to +\infty$. 
It seems difficult to obtain 
similar conclusions for $\lambda\in (\lambda_0,0)$
since $u_\lambda$ changes sign but, 
in general, if (\ref{H_2}) is satisfied for 
$u> C\ge 0$, then    
$\frac{\partial T}{\partial p} >0$ for $\lambda\ge C$. 
For example, this is the case in problem 
(\ref{concreto}), with $C= \sqrt[3]{\frac A{2K}}$.  

\end{enumerate}

\subsection*{Acknowledgement}

This work was partially supported by project 
UBACyT 20020120100029BA and PIP 11220130100006CO CONICET.

\end{document}